\documentclass[reqno]{amsart}
\usepackage{graphicx} 
\usepackage{mathpazo}
\usepackage{microtype}
\usepackage{tabularx}
\usepackage[utf8]{inputenc}
\usepackage{hyperref}
\usepackage{xcolor}
\usepackage{colortbl}  
\usepackage{amsmath,amsthm,amssymb,mathtools,bbm}
\usepackage{tikz}
\usetikzlibrary{arrows.meta,positioning,calc}
\usepackage[foot]{amsaddr}
\usepackage{graphicx}
\usepackage{upgreek}
\usepackage[normalem]{ulem} 
\usepackage[dvipsnames]{xcolor}
\usepackage{mathrsfs}
\usepackage{ytableau} 
\usepackage{tikz}
\usepackage{braids}
\usetikzlibrary{arrows}
\usetikzlibrary{braids}
\usepackage{float}
\usepackage[all,2cell]{xy}
\UseAllTwocells 
\usepackage{blkarray}

\newcommand{\ignore}[1]{}
\newcommand{\bra}[1]{\langle #1 |} 
\newcommand{\ket}[1]{| #1 \rangle}
\newcommand{\mset}[1]{\{\!\{#1\}\!\}}
\theoremstyle{definition}
\newtheorem{theorem}{Theorem}[section]
\newtheorem{proposition}[theorem]{Proposition}
\newtheorem{lemma}[theorem]{Lemma}
\newtheorem{definition}{Definition}[section]
\newcommand{\C}{\mathbb{C}}
\newcommand{\Q}{\mathbb{Q}}
\newcommand{\N}{\mathbb{N}}
\newcommand{\BrSys}{\mathsf{BrSys}}
\newcommand{\Sym}{\mathsf{Sym}}
\DeclareMathOperator{\End}{End}
\newtheorem{example}{Example}
\DeclareMathOperator{\Aut}{Aut}
\newcommand{\ppm}[1]{\textcolor{brown}{#1}}
\newcommand{\ppmm}[1]{{#1}}
\newcommand{\ecr}[1]{\textcolor{ForestGreen}{#1}}

\newcommand{\infeq}{$\infty$-equivalent}
\newcommand{\infeqe}{$\infty$-equivalence}
\newcommand{\CCwg}{CCwg}
\newcommand{\CCWG}{CCwg}

\begin{document}

\title[Equivalences among Constant YBOs]{Notions of equivalence among constant Yang--Baxter Operators:         \\        To $\infty$-Equivalence and Beyond}
 \author{Paul P. Martin $^1$}  
 \address{$^1$ School of Mathematics, University of Leeds}
\address{$^2$ Department of Mathematics, Texas A\&M University}
\author{Eric C. Rowell $^{1,2}$}
\begin{abstract}

 A fundamental question in any higher-representation classification problem is to determine an appropriate notion of equivalence among representations. In ordinary representation theory we have the notion of isomorphism, but in general there may be several coherent choices for morphisms and then the question arises of the relationship between corresponding notions of equivalence. This question is already relevant in the classification of constant Yang--Baxter operators where gauge equivalence (i.e., local basis change) is comfortingly familiar, but is by no means canonical. In this paper we first review the equivalences used in recent advances in this field (through charge-conserving YBOs); we determine the relationships between them; and then develop a toolkit of theorems applicable for testing equivalence in practice. 
 
 We use $\infty$-equivalence as the reference notion here: for $m\in\N$, $m$-equivalence denotes $n$-strand braid group representation equivalence for all $n \leq m$.  
This combines naturally with an appropriate variation of the ordinary Artin--Wedderburn Theorem. 
We show, for example, that for each dimension $N$ there is a finite $m(N)$ so that $m(N)$-equivalence implies $\infty$-equivalence, up to radical (but also show that the function $m(N)$ is unbounded).  In principle proving $\infty$-equivalence requires providing an infinite sequence of intertwiners, so this reduces the burden significantly. For another example, we show that in dimension $2$ all bar one local family of solutions are charge-conserving up to radical and $\infty$-equivalence --- a characterisation relatively simple enough that it gives us a framework for addressing the major open classification problem in dimension $3$.

\end{abstract}

\maketitle

\tableofcontents 

\section{An Enlightened Approach to Equivalence}

The dual-edged problem of 
{classifying} and constructing
complex matrix solutions to the 
{constant}
Yang-Baxter equation 
{is} interesting, yet computationally
daunting.  
Recent progress on the classification side through a holistic categorical approach sharpens the urge to find 
   useful 
   ways to 
   {organise solutions into equivalences classes}. 
   In particular the classification of charge conserving solutions \cite{MartinRowell} in all dimensions inspires the question: \emph{what fraction of solutions are \textbf{close to} charge conserving solutions, under any \textbf{reasonable} notion of equivalence?}  Of course \emph{reasonable} and \emph{close to} are open to interpretation. 

Classification is about completeness and about understanding. 
For the latter,
 an indigestibly large set of classes is unlikely to be of much {practical} use to 
{us.}  In this article we study several types of equivalence, attempting to be as broad-minded as possible. {This is in service to the goal of providing a palatable yet satisfying classification.}

\medskip 

\newcommand{\Braid}{Braid}
\newcommand{\Mat}{Mat}

A Yang--Baxter operator (YBO) is an (usually invertible) solution to the following matrix equation:
\begin{equation}\label{YBO}
    R_1R_2R_1=R_2R_1R_2
\end{equation} where $R\in\End(V\otimes V)$ for some finite dimensional vector space $V$, and $R_1:=R\otimes I_V$ and $R_2:=I_V\otimes R$.
Let $R$ and $S$ be two 
constant 
Yang-Baxter operators (YBOs), and denote by $\rho_n^R$ and $\rho_n^S$ the corresponding sequences of braid group representations, defined on generators by \[\rho_n^R(\sigma_i)=I_V^{\otimes i-1}\otimes R\otimes I_V^{\otimes n-i-1}\quad 1\leq i\leq n-1,\]  and similarly for $S$.
We say $R,S$ are  \textbf{$\infty$-equivalent} if there exists a sequence of 
invertible matrices $T_n$ 
so that $T_n\rho_n^R(\beta)=\rho_n^S(\beta)T_n$ for all $n\geq 2$ and $\beta\in B_n$.  In categorical language this is the same as saying that the corresponding strict, strong monoidal functors $F_R,F_S:\Braid\rightarrow \Mat$ 
{(also known as braid representations)}
are naturally (but not necessarily monoidally) equivalent \cite{maclane2013categories,adamek2004abstract}.  
Here $\Braid$ is Mac Lane's strict monoidal category of braids, monoidally generated by the elementary braid $\sigma \in \Braid(2,2)$ and its inverse $\sigma'$. 
From the perspective of representation theory, $\infty$-equivalence is the most appropriate notion 
of equivalence; however
 specifying
 an $\infty$-equivalence, in general, requires an infinite amount of information.  
On the other hand, many classification results in the literature (see, e.g. \cite{Hietarinta92,ESS}) consider equivalence up to local basis change together with some natural symmetries rather than $\infty$-equivalence, possibly for this reason.  In the set-theoretical community (see e.g., \cite{AkgunMerebVendramin2022}) these basis changes are usually restricted even further to relabelings, i.e., basis permutations.  We are confronted, as is often the case in mathematics, with a situation akin to the joke about the drunk man on his hands and knees looking for his house keys under a streetlight, as in Figure \ref{drunkkeys}.
A passer-by asks, “Did you lose your keys here?”
The drunk replies, “No, I lost them over there.”
“So why are you looking here?”
“Because the light is better here.”

\begin{figure}[ht]\includegraphics[width=9cm]{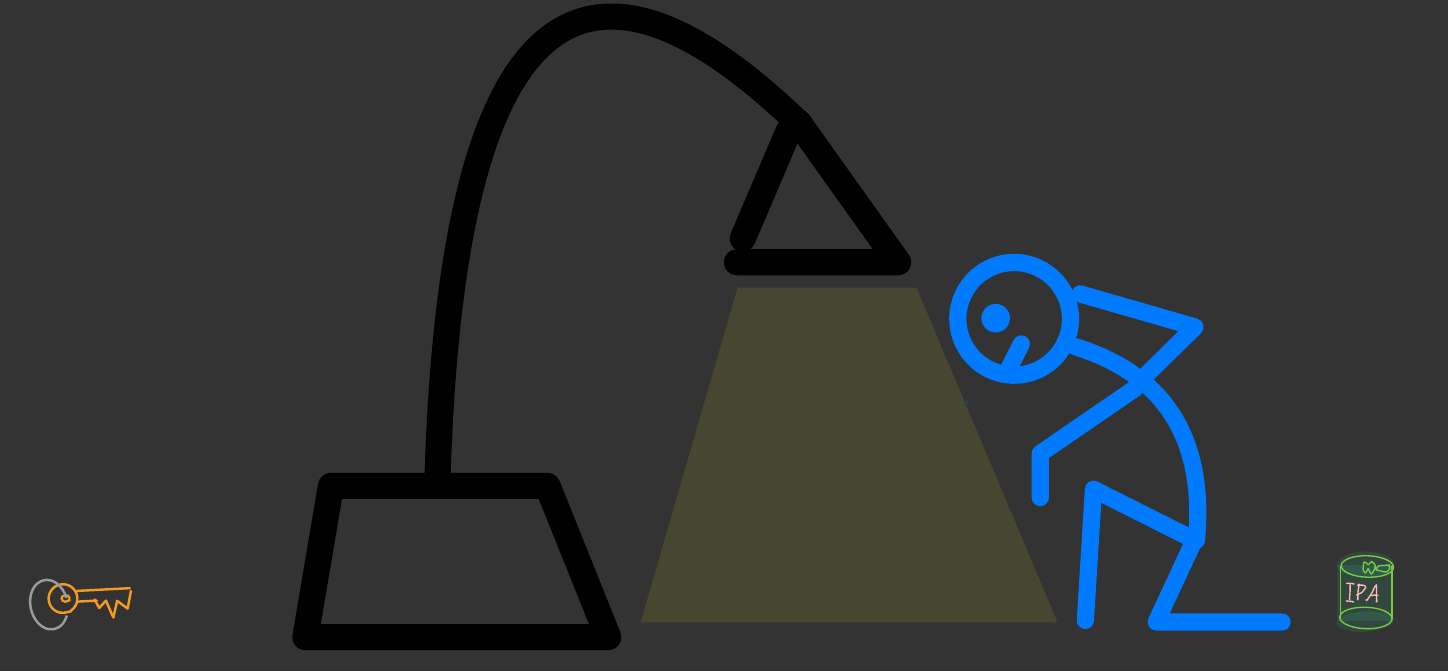}\caption{\label{drunkkeys} We sometimes find ourselves looking for solutions where the light is better, rather than exploring the {more likely, but more daunting,} shadows for the best answer.}\end{figure}

The main goal of this article is to shed some light on various forms of equivalence and the relationships amongst them.  We are particularly interested in ways to show two YBOs are $\infty$-equivalent.  In more detail this article achieves the following: 

\begin{enumerate}
  \item We provide a simple algorithm for determining when a YBO is locally equivalent to a CC YBO in Theorem \ref{thm:locallyCCcriterion}.
\item We give a proof that in rank-2 all but precisely one YBOs are \infeq\ to charge-conserving with glue \CCWG\ solutions:  
Theorem~\ref{thm:uniqueinftyclassinrank2}, {up to transpose symmetry.}
    \item We show that, globally, $3$-equivalence is insufficient for $\infty$-equivalence: Example \ref{ex:3notsufinfty}.
    \item We show that, for all $N$ there exists an $m(N)$ such that for semisimple YBOs $R,S$ $m(N)$-equivalence implies $\infty$-equivalence: Theorem \ref{thm:m-implies-infty-for-ss-rank-N}.
    \item We prove that for CC solutions $X$-equivalence implies $\infty$-equivalence: Theorem \ref{thm:Ximpliesinfty}
    \item We illustrate in Example \ref{ex:ACCexample} that in rank 3 \emph{additive charge conserving} YBOs \cite{HMR} may fail to be $\infty$-equivalent to a \CCWG\ YBO.
    \item We show that 1-step $DS$-moves are not automatically transitive in Example \ref{ex:texas2step}.
    \item We describe a more general Morita-like form of equivalence inspired by \cite{RWlocalization}, illustrated with two solutions of distinct ranks that are nonetheless equivalent in \ref{ex:BrSysequiv}
  
\end{enumerate}

\textbf{Acknowledgements} The work of PPM was partially supported by the EPSRC Programme Grant 
EP/W007509/1.  The research of ECR was partially supported by a Royal Society Wolfson Visiting Fellowship and US NSF grant DMS-2205962.

\section{Equivalences and Symmetries of Many Stripes}

\newcommand{\breath}{\medskip \newline}

{In this section we discuss the organisational tools available in the classification of constant YBOs.} 

{In principle these tools are analogous to the tools available in the classification of representations of, say, a group. In practice, higher representation theory is less canonical than ordinary representation theory, so all of the following analogues require great care in upscaling, but the analogy is still useful  for guiding us through this section. Thus an initial organisation of representations of a group can be by field, and then dimension; further organisation is achieved by grouping representations into isomorphism classes; there may be  simple invertible transformations between representations that do not preserve isomorphism, but still collect representations into natural orbits; and it may even be useful to group according to whether two representations are identically equal maps (for example representations defined by different constructions).}

{In group representation theory it is, ultimately, isomorphism which does most of the work. There is no canonical upscaling of this notion. Instead both $\infty$-equivalence and local equivalence are analogues - with local equivalence  stricter and easier to verify, but $\infty$-equivalence producing a much more manageable set of classes. In \S\ref{ss:m} we break down testing for $\infty$-equivalence into steps more like local equivalence.}

{In the same spirit of providing easier steps towards $\infty$-equivalence, this time sufficient but not necessary, in \S\ref{ss:DSL} we} {describe a method, due to Doiku and Smoktunowitz \cite{DS}, for producing $\infty$-equivalences to a given solution based on its local automorphism group.} 

{In \S\ref{ss:Gal} we give examples of `symmetry'. Easy invertible transformations that construct new representations from old, but do not in general preserve $\infty$-equivalence.}

\medskip 

As noted  in the Introduction 
above, 
in categorical terms a constant YBO is 
effectively the same data as  
\emph{a braid representation} - a strict monoidal functor $F:\Braid \rightarrow \Mat$. 
In these terms, the first and most fundamental layer of classification is 
not so challenging - 
simply by the 
{\em rank}, the  
image of the object 1. Here we write $N$ for $F(1)$. 
Thus a rank-$N$   YBO is a $N^2 \times N^2 $ matrix. 
It is immediate that different rank implies different $\infty$-equivalence class. 
We defined rank as the image $F(1) =N$; but the term rank is sometimes used for the $n$ in $B_n$. So if there is risk of ambiguity we will refer to $N$ as big-rank and $n$ as little-rank. 

\medskip 
\subsection{$m$-Equivalence}  \label{ss:m}
A kind of poor man's version of $\infty$-equivalence is \textbf{$m$-equivalence} \cite{MRT25} {for some $m \in \N$}:  
we say $R$ and $T$ are $m$-equivalent if
there exists a set of 
invertible matrices $\{T_n: 1\leq n \leq m\}$ 
so that $T_n\rho_n^R(\beta)=\rho_n^S(\beta)T_n$ for all $\beta\in B_n$ 
for $n \in 1,2,\ldots,m$. 
(Note that this yields \infeqe\ in the limit.) 

Observe that 1-equivalence is essentially the rank condition (we have no $R$ matrix 
in this little rank - the braid group is the trivial group).
And obviously $m$-equivalence implies $m\!-\!1$-equivalence, meanwhile the converse is certainly false for small $m$. 
Then 2-equivalence requires that $R$ and $S$ have the same 
Jordan form. 
If our YBO's are unitary (or otherwise diagonalisable) then this says simply that they have the same spectrum. 
And it will be clear that different spectra implies not 2-equivalent, so this is a very 
convenient next test of equivalence between 1-equivalence and 2-equivalence.  

\medskip 

The reader's curiosity will doubtless have led them to wonder if there is some finite $m$ such that $m$-equivalence implies $\infty$-equivalence (see also \cite[Paragraphs 7.2-7.3]{MRT25}), as this would render the burden of specifying an $\infty$-equivalence a finite list of $T_n${'s}.  Counterexamples for $m=2$ are found in \emph{loc. cit.} and in \cite{LPW}, where a conjecture in the involutive case is attributed to \cite{Gur}.  We see that $m=3$ is also insufficient in 
the following:

\begin{example}\label{ex:3notsufinfty}

Consider the rank $6$ involutive solutions $R$ and $S$ associated with the pairs of Young diagrams $([2],[1,1,1,1])$ and $([1,1,1,1],[2])$ as in \cite{LPW}.
Unpacking their notation 
we have the following explicit descriptions of $R$ and $S$.  Let $\ket{0},\ldots,\ket{5}$ be a basis of $\C^6$, and extend to $\C^6\otimes \C^6$.  Then, $R$ acts by $1$ on $\ket{00},\ket{01},\ket{10}$ and $\ket{11}$, by $-1$ on $\ket{22},\ldots,\ket{55}$ and by $\ket{ij}\mapsto \ket{ji}$ on all other basis elements.  Similarly, $S$   acts by $1$ on $\ket{00},\ldots,\ket{33}$ by $-1$ on $\ket{44},\ket{45},\ket{54}$ and $\ket{55}$ and by $\ket{ij}\mapsto \ket{ji}$ on all other basis elements.
 
 Since $R$ and $S$ are involutive, $\rho_n^R$ and $\rho_n^S$ factor over the symmetric group $\Sigma_n$.  
    Thus to check that $\rho_3^R$ and $\rho_3^S$ are equivalent we may simply compute traces for the images of the 3 conjugacy classes of $\Sigma_3$, for which we obtain $216,0,12$ 
    (trivial, 2-cycle, 3-cycle respectively)
    for both $R$ and $S$.  On the other hand, the general results of \cite{LPW} show that $R$ and $S$ are not  $\infty$-equivalent.  In fact, one computes that for $n=4$ the traces disagree on the conjugacy class of $4$-cycles--one gets $12$ and $-12$, respectively.
\end{example}

More generally, it can be shown that for all $m\geq 1$, $m$-equivalence does not imply $m+1$-equivalence.  
Inspired by the above, we seek to
construct, for each $m$, a pair of {involutive} unitary solutions $R,S$ so that the characters of $\rho_m^R$ and $\rho_m^S$ agree, but those of $\rho_{m+1}^R$ and $\rho_{m+1}^S$ differ. {As above we use involutive YBOs $R,S$ so that  
each $B_m$ representation is an $\Sigma_m$ representation.
We observe that it} is enough to check the agreement on $k$ cycles for $k\leq m$ and disagreement on $(m+1)$-cycles. {The essential point to this observation is 
{firstly to recall that every permutation can be expresses as a product of commuting cycles; then}
to recall 
that the image of a product of disjoint cycles can be conjugated to a product that has a tensor decomposition.  Indeed, suppose 
{for example}
that $a,b$ are disjoint $k_1$ and $k_2$ cycles. We may conjugate $ab$ to a canonical product \[a'b'=(1\;2\;\cdots\; k_1)((k_1+1)\;\cdots (k_1+k_2)).
\] 
{This is $ (1\;2\;\cdots\; k_1) \otimes (1\;2\;\cdots\; k_2) $ in the $\Sym$ monoidal category.}
Then $\rho_m^R(a'b')=\rho^R_{k_1}(a')\otimes\rho_{m-k_1}^R(b')$, so $tr(\rho_m^R(a'b'))=tr(\rho^R_{k_1}(a'))\cdot tr(\rho_{m-k_1}^R(b'))$ by properties of 
the Kronecker 
$\otimes$.  
{Generalising in the obvious way, }
the character of $\rho_m^R$ is determined on cycles, hence the claim.
} 

The counterexamples {that we report here} are unlikely to be particularly illuminating as they are constructed using an LLM, but we provide them for completeness' sake.  
The {sought} offending pairs $(R,S)$ {at each $m$} can be conveniently 
described by means of a pair of Young diagrams {each} using \cite{LPW}, explicitly: write \(j^a\) for \(a\) rows of length \(j\), and let \(\varnothing\) denote the empty diagram. For fixed $m$ take the two pairs
\[
(\lambda_m,\varnothing),
\qquad
(\mu_m,\varnothing),
\]
where
\[
\lambda_m
=\bigl(j^{\binom{m+1}{j}}:
1\leq j\leq m+1,\ j\ \text{even}\bigr),
\qquad
\mu_m
=\bigl(j^{\binom{m+1}{j}}:
1\leq j\leq m+1,\ j\ \text{odd}\bigr),
\]
with rows arranged in decreasing order. Both diagrams have size
\[
N_m=(m+1)2^{m-1}.
\]
The first examples are
\[
\begin{array}{c|c|c|c}
m&\lambda_m&\mu_m&N_m\\ \hline
1&(2)&(1^2)&2\\
2&(2^3)&(3,1^3)&6\\
3&(4,2^6)&(3^4,1^4)&16\\
4&(4^5,2^{10})&(5,3^{10},1^5)&40\\
5&(6,4^{15},2^{15})&(5^6,3^{20},1^6)&96
\end{array}
\]
Indeed, {by direct calculation}, the difference of the corresponding \(k\)-cycle traces is
\begin{equation}   \label{eq:kcyclediff}
tr(\rho^R_m((1\;\cdots\;k))) - tr(\rho^S_m((1\;\cdots\;k))) \; = \; \sum_{j=0}^{m+1}(-1)^j\binom{m+1}{j}j^k,
\end{equation}  
which vanishes for \(1\leq k\leq m\), but equals
\[
(-1)^{m+1}(m+1)!
\]
when \(k=m+1\).

\medskip
With this evidence in hand, the reader might overcome their disappointment and update their speculation: for a fixed $N$, is there an $m(N)$ so that $m(N)$-equivalence implies $\infty$-equivalence for any pair of rank $N$ YBOs?  This time our optimism is rewarded, at least partially:

\begin{theorem}\label{thm:m-implies-infty-for-ss-rank-N}  Fix $N\in\N$.  There exists an integer $m(N)$ such that if $R$ and $S$ are $m(N)$-equivalent rank $N$ YBOs, with  $\rho_n^R$ and $\rho_n^S$ semisimple for each $n$, then $R$ and $S$ are $\infty$-equivalent.

\end{theorem}
In fact we will prove something slightly stronger, but this will be revealed later.
\begin{proof}

  First, fix $N$ and two arrays of $N^4$ variables $r_{ij}$ and $s_{ij}$ where $1\leq i,j\leq N^2$.    Consider the polynomial ring in $2N^4$ variables $K=\C[r_{ij},s_{ij}: 1\leq i,j\leq N^2]$, and the ideal $\frak{A}$ generated by the entries of 
  {the braid anomalies $R_1 R_2 R_1 -R_2 R_1 R_2$ (cf.}
   equation \eqref{YBO})
  for both 
  {arrays} 
  $S$ and $R$.
  Set $\mathscr{A}=K/\frak{A}$, i.e., the coordinate ring the variety of pairs of $N^2\times N^2$ Yang--Baxter operators.   
  Allowing ourselves a bit of foreshadowing, we note that $\mathscr{A}$ is N\"otherian, by the Hilbert Basis theorem.  Using the notation as above we define a collection of polynomials in $\mathscr{A}$, one for each $n\geq 1$ and $\beta\in B_n$:
  \[F_{n,\beta}=Tr(\rho_n^R(\beta))-Tr(\rho_n^S(\beta)).\]  Now let $\frak{B}$ be the ideal in $\mathscr{A}$ generated by all $F_{n,\beta}$--this is ideal corresponding to pairs of solutions (or points on the variety) for which the characters agree.  
  Since $\mathscr{A}$ is N\"otherian, $\frak{B}$ has finitely many generators, which may be assumed to be of the form $F_{n_i,\beta_i}$.  Let $m(N)$ be the maximum of the $n_i$.  Clearly $\frak{B}$ is generated by the (infinite collection of) polynomials $F_{n,\beta}$ with $1\leq n\leq m(N)$.

  In particular, we see that if the characters of $\rho_n^R$ and $\rho_n^S$ agree for all $1\leq n\leq m(N)$ then the characters of $\rho_n^R$ and $\rho_n^S$ agree \emph{for all} $n\geq 1$. This statement does not require semisimplicity.

Now recall that semisimple representations are determined by their characters, so that if $R$ and $S$ are $m(N)$-equivalent then $\rho_n^R$ and $\rho_n^S$ have the same characters for $1\leq n\leq m(N)$.  Thus their characters agree for all $n$, so we have all $n$ equivalence, i.e. $\infty$-equivalence.
\end{proof}

\subsection{DS(L) Equivalence}  \label{ss:DSL}
Two notions of equivalence that require only a finite amount of information to specify, yet imply $\infty$-equivalence are local equivalence $L$ and DS-equivalence.  These are as follows: $R$ and $S$ are \textbf{locally equivalent} if there exists an invertible matrix $A$ with $(A\otimes A)R=S(A\otimes A)$.  We will write such local equivalences as $L_A(R)=(A\otimes A)R(A\otimes A)^{-1}.$ The categorical interpretation is that of a \emph{monoidal} natural equivalence between the corresponding functors.  In this case the $\infty$-equivalence is achieved by $A^{\otimes n}\rho_n^R=\rho_n^SA^{\otimes n}$.  

DS-equivalence is more subtle, as 
{the construction of the class of $R$}
depends on the YBO itself: suppose $T$ is invertible matrix such that $T\otimes T$ commutes with $R$. Then  $(I\otimes T)R(I\otimes T)^{-1}$ and $(T\otimes I)R(T\otimes I)^{-1}$ both satisfy the YBE again, and yield $\infty$-equivalent YBOs. We shall denote the latter transformation by $DS_T(R)=(T\otimes I)R(T\otimes I)^{-1}$ which will be called a \textbf{DS-move}. Then we say that $R$ and $S$ are \textbf{DS-equivalent} (after Doiku and Smoktunowitz \cite{DS}) if there is a sequence of DS-moves carrying $R$ to $S$. For a single DS-move $S=DS_T(R)$ the corresponding $\infty$-equivalence is given by $T_n:=I\otimes T\otimes\cdots\otimes T^{\otimes (n-1)}.$  
{Observe that DS-equivalence is indeed an equivalence relation - the DS-move relation is symmetric and reflexive, and the sequence gives transitive closure.}

It can of course happen that two YBOs are related by a composition of local and DS-equivalences, while being neither locally equivalent or DS-equivalent.  This can happen, for example, if there are two invertible matrices $A,T$ such that $T\otimes T$ commutes with $R$ and $(A\otimes AT)R=S(A\otimes AT)$.  Then the corresponding $\infty$-equivalence is given by $T_n=A^{\otimes n}(I\otimes T\otimes\cdots\otimes T^{n-1})$.  We will call such moves {(1-step)} \textbf{DS-local}, or DSL for short. {As above, we will say that $R$ and $S$ are \textbf{DSL-equivalent} if there is a sequence of DSL-moves relating $R$ to $S$.} It is conceivable that two YBOs require a multi-step chain of DSL moves.  We will see that this is the case in Example \ref{ex:texas2step}.  On the other hand, if $R$ and $S$ are related by a sequence of 1-step DSL moves, the equivalence can be accomplished by first applying a sequence of DS-moves and then applying a local equivalence at the end.  The reason is the following easy:
\begin{lemma}
    Let $R$ be an $N^2\times N^2$ YBO and $A\in GL_N$.  Then $T\otimes T$ commutes with $$A\otimes AR(A\otimes A)^{-1}$$ if and only if $A^{-1}TA\otimes A^{-1}TA$ commutes with $R$.
\end{lemma}
\begin{proof}
    If \[[T\otimes T,A\otimes AR(A\otimes A)^{-1}]=0\] then 
    \[(A\otimes A)^{-1}[T\otimes T,A\otimes AR(A\otimes A)^{-1}](A\otimes A)=[A^{-1}TA\otimes A^{-1}TA,R]=0.\]
\end{proof}
As a consequence, if $T\otimes T$ commutes with $R$ then $L_A\circ DS_T(R)=DS_{A^{-1}TA}\circ L_A(R)$ (and $DS_{A^{-1}TA}$ is a valid DS-move for $L_A(R)$).
In particular, if $S$ is obtained from $R$ as $S=L_{A_n}DS_{T_n}\circ\cdots\circ L_{A_1}DS_{T_1}(R)$ with each $T_i\otimes T_i$ commuting with the previous iteration, then we may write this as $S=L_{B}\circ DS_{M_n}\circ \cdots \circ DS_{M_1}(R)$, that is, we may achieve this transformation by performing a sequence of DS-moves followed by a single local move.

As we mentioned in the introduction $\infty$-equivalences may be of a particularly simple type: DSL.  Here we give an example, {noting that the original proof in \cite{AlmateariMartinRowell} uses an inductive construction of the intertwiners.

\begin{example}\label{ex:aDSLequivrank2}
    Let \[R= \left[ \begin {array}{cccc} 1&0&0&1\\ \noalign{\medskip}0&0&1&0
\\ \noalign{\medskip}0&1&0&0\\ \noalign{\medskip}0&0&0&1\end {array}
 \right] ,\quad S=\left[ \begin {array}{cccc} 1&0&0&1\\ \noalign{\medskip}0&0
&-1&0\\ \noalign{\medskip}0&-1&0&0\\ \noalign{\medskip}0&0&0&1
\end {array} \right]. \]
These are examples of charge-conserving-with-glue YBOs.  Next define 

\[T:= \left[ \begin {array}{cc} -1&0\\ \noalign{\medskip}0&1\end {array}
 \right],\quad A:= \left[ \begin {array}{cc} i&0\\ \noalign{\medskip}0&1\end {array}
 \right] 
 .\]
One checks that $T\otimes T$ commutes with $R$.  Moreover, $Z:=(A\otimes AT) = diag(1,i,-i,1)
$ satisfies $ZR=RZ$.  Thus $R$ and $S$ are DSLly equivalent.  These lie in distinct local classes, cf. \cite{Hietarinta92}.
\end{example}

 For a slightly more thrilling example consider the following:

\begin{example}\label{ex:rank3gaussian}
Define parameters
\[
  \eta=e^{\pi i/6}=\frac{\sqrt3+i}{2},
  \;
  \omega=\eta^4=e^{2\pi i/3}, \;  x=-\frac{\eta}{\sqrt3},
  \;
  y=1+x.
\]
 On \(\mathbb C^3\), and for $j\in \mathbb Z/3\mathbb Z$, let
\[
  Xe_j=e_{j+1},\; Ze_j=\omega^j e_j,
 \; \text{and} \; P=X\otimes Z.
\] 
Here $X,Z$ are the familiar Pauli matrices from quantum information science. 

The two rank \(3\) Gaussian Yang--Baxter operators from
\cite[Example 7.5]{MRT25} may be written compactly as
\[
  R=xI_9+yP_R+yP_R^2,
  \qquad
  \mbox{ {where} } \qquad 
  P_R=XZ^2\otimes XZ,
\]
and
\[
  S=xI_9+yP_S+yP_S^2,
  \qquad
  P_S=X\otimes Z.
\] It was shown, using a Temperley-Lieb algebra representation argument, that these 
 two operators are \(\infty\)-equivalent. A more elementary direct proof is as follows: define
\[
  A=
  \frac1{\sqrt3}
  \begin{bmatrix}
  i&i&\eta^{-1}\\
  \eta^{-1}&-\eta&\eta^{-1}\\
  \eta^{-1}&i&i
  \end{bmatrix},
  \qquad
  T=
  \begin{bmatrix}
  0&\omega&0\\
  \omega^2&0&0\\
  0&0&1
  \end{bmatrix}.
\]
Then
\[
  (T\otimes T)S=S(T\otimes T)
\]
and, with \(B=A\otimes(AT)\),
\(
  BSB^{-1}=R.
\)  We remark that $R$ and $S$ are not locally equivalent.

Incidentally, both \(A\) and \(T\) are one-qutrit \emph{Clifford} matrices, that is, they normalize the group of 1-qutrit Pauli gates $\{\omega^aX^jZ^k:0\leq a,j,k\leq 2\}$. 
\end{example}

The following shows that 2-step DS(L) equivalences may sometime be unachievable by means of 1-step DSL equivalences:

\begin{example}\label{ex:texas2step}
Define 
\begin{align*}
D =
\begin{bmatrix}
-1 & 0 & 0 & 0\\
0 & 0 & 1 & 0\\
0 & 1 & 0 & 0\\
0 & 0 & 0 & -1
\end{bmatrix}, \quad
T =
\begin{bmatrix}
1 & 0\\
0 & -1
\end{bmatrix},
\quad \text{and}\quad
S =
\begin{bmatrix}
1 & 1\\
0 & 1
\end{bmatrix}.
\end{align*}

One finds that $T\otimes T$ commutes with $D$, so we obtain another solution
\[
R'=DS_T(D)=(T\otimes I)\,D\,(T\otimes I)^{-1}=\begin{bmatrix}
-1 & 0 & 0 & 0\\
0 & 0 & -1 & 0\\
0 & -1 & 0 & 0\\
0 & 0 & 0 & -1
\end{bmatrix},
\] i.e., a scalar multiple of the ``flip" matrix.
Since any $2\times 2$ matrix $M$ has $[M\otimes M,R']=0$ any $DS_M(R')$ is a YBO.
In particular,
\[
R''=DS_S(R')=(S\otimes I)\,R'\,(S\otimes I)^{-1}
=
\begin{bmatrix}
-1 & -1 &  1 &  1\\
 0 &  0 & -1 & -1\\
 0 & -1 &  0 &  1\\
 0 &  0 &  0 & -1
\end{bmatrix}
\]
is a YBO. Thus $D$ and $R''$ are related by 2 DS-moves.

Now we look for a $1$-step DSL-move between $D$ and $R''$, i.e., a solution to  the system
\[
(M\otimes M)D=D(M\otimes M),\;
R''=(CM\otimes C)\,D\,(CM\otimes C)^{-1},\;
\det(C)\det(M)\neq 0, 
\]
with $C,M$ both $2\times 2$ matrices.  A simple Gr\"oebner basis calculation shows this is impossible.
\end{example}

\subsection{Galois Symmetry}  \label{ss:Gal}
 Suppose that $R$ is a YBO defined over some field extension $K$ of $\Q$, and $\varphi\in\Aut(K/\Q)$ is a field automorphism of $K$ fixing $\Q$.  Then plainly $\varphi(R)$ remains a YBO.  Of course $R$ and $\varphi(R)$ will not typically be $\infty$-equivalent or even $2$-equivalent, but one can certainly be recovered from the other.  We will say that solutions $R$ and $S$ are \textbf{Galois equivalent} if there is a field automorphism $\varphi$ of the type described above carrying $R$ to $S$.  Let us enjoy the following example together:
\begin{example}
   Continue with the notation as above and set $P=P_S=X\otimes Z$.  Unitary YBOs of the form $a\cdot I+b\cdot P+c\cdot P^2$ have appeared in the literature, going back at least to Goldschmidt and Jones \cite{GJ}, and are sometimes called Gaussian solutions \cite{GalindoRowell}.  

Let us describe all such solutions.  Firstly, one finds that there are no solutions with $a=0$. Thus we may normalise to assume $a=1$, understanding that an overall scalar factor can be introduced later to ensure unitarity.  We define $R(b,c)=I+b\cdot P+c\cdot P^2$.  With this we find $6$ nontrivial solutions, as the YBE is equivalent to the simultaneous vanishing of $c^3 - 1$ and $bc^2 + b^2 + c$.  The solutions are given by: \[(b,c)\in \{(\omega^i,\omega^i),(1,\omega^i),(\omega^i,1)\}\] for $i=1,2$.  In particular, all solutions are in the field $\Q(\omega)$. 

We explore these six solutions as follows: first, we use scalar multiplication and Galois conjugation to ensure obtain solutions that have identical eigenvalues: $1$ with multiplicity $3$ and $\omega$ with multiplicity $6$.  

Next we look for local equivalences among these solutions.  We find that this yields two non-locally equivalent classes, represented by $Ra:=\frac{i}{\sqrt{3}}R(\omega^2,\omega^2)$ and $Rb:=\frac{i}{\sqrt{3}}\varphi(R(\omega,\omega))$ where $\varphi:\omega\mapsto\omega^{-1}$ determines $\varphi$.  Miraculously, $Ra$ and $Rb$ are $DS$-equivalent: Setting $A=\left[ \begin {array}{ccc} 1&0&0\\ \noalign{\medskip}0&0&1
\\ \noalign{\medskip}0&1&0\end {array} \right] 
$ we find that $(A\otimes I)Ra=Rb(A\otimes I)$.
In particular we find that up to rescaling, Galois and $\infty$-equivalence this is the only solution of this form.  Of course, the solutions of Example \ref{ex:rank3gaussian} are in this class.
  \end{example}

The rank $3$ solution $Ra$ is analogous to the lone Hietarinta solution in rank $2$ that is not $\infty$-equivalent to any other of Hieterinta's solutions.  It is an interesting question to define and classify the appropriate class of solutions in all ranks.

\section{Is Your Yang-Baxter Operator Secretly Charge-Conserving?}
{As a complete classification of charge-conserving solutions is available \cite{MartinRowell}, we would be fool{ish} not to ask how broad a swath they cut 
{through the general problem, up to the various known symmetries, including} 
both locally and up to $\infty$-equivalence. }
{In this section we first give an algorithm for determining whether a given braid representation is locally to a charge-conserving (CC) representation and apply it to well-known non-CC representations.}  Then we provide a proof that the YBO associated with the 8 vertex model \cite{Hietarinta92} is $\infty$-equivalent (but not locally equivalent) to a CC YBO.  Finally we prove a recent conjecture that conjugation by a diagonal matrix, a known symmetry of the set of CC YBOs \cite{MartinRowell}, implies $\infty$-equivalence. 

\medskip 

Denote the standard ordered basis for $V:=\C^N$ by $\{\ket{i}:1\leq i\leq N\}$ and extend to $V^{\otimes k}$ as $\{\ket{i_1\cdots i_k}:1\leq i_j\leq N\}$.  The \emph{charge} of such a basis vector $\ket{i_1\cdots i_k}$ is the multiset $\mset{i_1,\ldots,i_k}$ of indices. Let $V^{\mset{i_1,\ldots,i_k}}$ denote the span of all basis vectors with charge $\mset{i_1,\ldots,i_k}$. 
An $N^k\times N^k$ matrix $M$ is 
{rank-$N$}
\textbf{charge conserving} if it preserves charge, that is, if $V^{\mset{i_1,\ldots,i_k}}$ is $M$-invariant. {Structurally, such an $M$ is block diagonal with respect to a basis grouped by charge.} A charge-conserving, or CC, Yang-Baxter operator is a Yang-Baxter operator that is charge-conserving as a $N^2\times N^2$.  

{Fix $N$.}
We say that two operators $L,M$ on $V^{\otimes k}$ are \textbf{locally equivalent} if there is an invertible operator $A$ on $V$ such that $A^{\otimes k}L(A^{-1})^{\otimes k}=M$, generalising the $k=2$ 
definition above.

\subsection{Locally Equivalent to Charge-Conserving}
We have another characterisation of charge-conserving matrices, which will shortly be turned into a tight little algorithm. Fix $k,N \in \N$. 
   For an $N\times N$ matrix $A$ define $A_i=I\otimes\cdots\otimes A\otimes\cdots \otimes I$ ($k$ factors) where the $A$ appears in the $i$th position. Next define 
   $$
   \Delta_k(A)=\sum_{i=1}^kA_i    . 
   $$ 

\begin{lemma}  Let $N,k$ be integers.
   An $N^k\times N^k$ matrix $M$ is rank $N$ charge conserving if and only if $[M,\Delta_k(H)]=0$ for all diagonal $N\times N$ matrices $H$.
\end{lemma}
\begin{proof}
   Let $H_{ij}=\delta_{ij}h_j$ be a diagonal matrix. Since $\Delta_k(H)\ket{i_1\cdots i_k}=\sum_{j=1}^kh_{i_j}\ket{i_1\cdots i_k}$, we see that $\Delta_k(H)$ is constant on $V^{\mset{i_1,\ldots,i_k}}$. Thus, if $M$ is charge conserving, $M$ commutes with $\Delta_k(H)$.  On the other hand, if $M$ is not charge conserving so that $\bra{t_1\cdots t_k}M\ket{i_1\cdots i_k}\neq 0$ for some $\mset{t_1,\ldots,t_k}\neq \mset{i_1,\ldots,i_k}$ we may find a diagonal $H$ so that $\sum_{j=1}^kh_{i_j}\neq \sum_{j=1}^kh_{t_j}$, so that $\Delta_k(H)$ does not commute with $M$.
\end{proof} 

The key is that, as $H$ ranges over all diagonal matrices,
the quantities $\sum_{j=1}^k h_{i_j}$ jointly distinguish
the multiset subspaces $V^{\mset{i_1,\ldots,i_k}}$.
This characterization suggests a test for determining when a matrix $S$ is locally equivalent to a charge conserving matrix $R$, as follows: 

\newcommand{\enzuigiri}{enzuigiri}  

\begin{theorem}\label{thm:locallyCCcriterion}
{(\enzuigiri\ method)}
Fix positive integers $k,N$.
An $N^k\times N^k$ complex matrix $S$ is locally equivalent
to a charge-conserving matrix if and only if there exists
an $N\times N$ complex matrix $M$ such that:
\begin{enumerate}
\item $[S,\Delta_k(M)]=0$;
\item $M$ has $N$ distinct eigenvalues
      $\lambda_1,\ldots,\lambda_N$;
\item for any nonnegative integer vectors
      $\mathbf{n}=(n_1,\ldots,n_N)$ and
      $\mathbf{m}=(m_1,\ldots,m_N)$ satisfying
      $\sum_i n_i=\sum_i m_i=k$, we have
      \[
      \sum_{i=1}^N n_i\lambda_i
      =
      \sum_{i=1}^N m_i\lambda_i
      \quad\Longrightarrow\quad
      \mathbf{n}=\mathbf{m}.
      \]
\end{enumerate}
\end{theorem}

\begin{proof}
Suppose such an $M$ exists. Since its eigenvalues are
distinct, we may write $M=A^{-1}HA$, where
$H=\operatorname{diag}(\lambda_1,\ldots,\lambda_N)$.
Set
\[
T=A^{\otimes k}S(A^{\otimes k})^{-1}.
\]
Then $[T,\Delta_k(H)]=0$.

For each nonnegative integer vector $\mathbf{n}$ with
$\sum_i n_i=k$, let $V_{\mathbf{n}}$ be the span of the
standard basis tensors in which the index $i$ occurs
exactly $n_i$ times.
The operator $\Delta_k(H)$ acts on $V_{\mathbf{n}}$
as the scalar $\sum_i n_i\lambda_i$.
By condition~(3), these scalars are distinct for distinct
$\mathbf{n}$. Thus the spaces $V_{\mathbf{n}}$ are
precisely the eigenspaces of $\Delta_k(H)$.
Consequently, $T$ preserves each $V_{\mathbf{n}}$,
so $T$ is charge conserving.

Conversely, suppose
$T=A^{\otimes k}S(A^{\otimes k})^{-1}$
is charge conserving. Choose
\[
\lambda_i=(k+1)^{i-1},
\qquad
H=\operatorname{diag}(\lambda_1,\ldots,\lambda_N).
\]
Uniqueness of base-$(k+1)$ expansions implies
condition~(3), since $0\le n_i\le k$.
As $T$ preserves every $V_{\mathbf{n}}$ and
$\Delta_k(H)$ acts as a scalar on each such space,
we have $[T,\Delta_k(H)]=0$.
Therefore $M=A^{-1}HA$ satisfies all three conditions.
\end{proof}

{The method is thus to determine the set of matrices $M$ satisfying (1) - an essentially linear problem - and then search for solutions to (2,3).}
The efficacy of this method is illustrated in the following example:

\begin{example}\label{ex:rank3gaussiannotcc}
  {Recall the notation of Example \ref{ex:rank3gaussian} above} Define 
\[R=\frac{1}{\sqrt{3}}(I+\omega(P+P^2)) . \] 
This is indeed 
a unitary YBO (see, e.g. \cite{RWlocalization}).  
It is straightforward to verify that the only invertible $3\times 3$ matrices $M$ such that $M\otimes I+I\otimes M$ commutes with $R$ are scalar multiples of the identity.  Thus $R$ is not locally equivalent to a charge conserving matrix.
\end{example}

In the other direction it is not hard to test the utility of this method in finding a charge conserving solution equivalent to a given solution.  Although somewhat artificial the following meets the goal:

\begin{example}
  Define $R=A\otimes AS(A\otimes A)^{-1}$ where  $A= \left[ \begin {array}{ccc} 1&1&1\\ \noalign{\medskip}0&1&1
\\ \noalign{\medskip}0&0&1\end {array} \right] 
$ and \[S= \left[ \begin {array}{ccccccccc} a&0&0&0&0&0&0&0&0
\\ \noalign{\medskip}0&a+b&0&a&0&0&0&0&0\\ \noalign{\medskip}0&0&a&0&0
&0&0&0&0\\ \noalign{\medskip}0&-b&0&0&0&0&0&0&0\\ \noalign{\medskip}0&0
&0&0&b&0&0&0&0\\ \noalign{\medskip}0&0&0&0&0&0&0&a&0
\\ \noalign{\medskip}0&0&0&0&0&0&a&0&0\\ \noalign{\medskip}0&0&0&0&0&-
b&0&a+b&0\\ \noalign{\medskip}0&0&0&0&0&0&0&0&a\end {array} \right].
\] 
{By construction this $R$ is obviously $\infty$-equivalent to a CC one, but is not itself CC. We will use the \enzuigiri\ method of Theorem \ref{thm:locallyCCcriterion} on $R$ to recover an equivalent CC.}

Then one computes the family of 
{matrices}
$M$ such that $[\Delta_2(M),R]=0$: these are, for arbitrary $u,v,w$ of the form $M=\left[ \begin {array}{ccc} u & v-u & w-v\\
0 & v & w-v\\
0 & 0 & w\end {array} \right],$ so choose $u,v,w$ distinct such that $u+v,u+w,v+w$ are distinct. 
Taking a basis of eigenvalues of such an $M$ provides a suitable local basis change that renders $R$ charge conserving.
\end{example}

\subsection{$\infty$-Equivalent to Charge-Conserving}
As we have mentioned, finding $\infty$-equivalences can be challenging.  One key result of this paper provides such an equivalence:
We show 
that the 
so-called $8$ vertex Yang-Baxter operator in rank $2$ is $\infty$-equivalent to a charge conserving solution. {(It is a two-parameter class of solutions, but we treat them all together, and will speak of them as if a single solution.)} 

In its original form the $8$ vertex solution is (cf. \cite{Hietarinta92} for the more classical ``star-triangle" form):  
\[
\left[ \begin {array}{cccc} {p}^{2}+2\,pq-{q}^{2}&0&0&{p}^{2}-{q}^{2}
\\ \noalign{\medskip}0&{p}^{2}-{q}^{2}&{p}^{2}+{q}^{2}&0
\\ \noalign{\medskip}0&{p}^{2}+{q}^{2}&{p}^{2}-{q}^{2}&0
\\ \noalign{\medskip}{p}^{2}-{q}^{2}&0&0&{p}^{2}-2\,pq-{q}^{2}
\end {array} \right]. 
\]
{where  $p,q$ are non-zero to ensure invertibility.}
Dividing by $2p^2$ and then substituting $x=q/p$ we obtain the 
formulation 
$S$ in the following: 

\begin{proposition}\label{prop:8va}
Let $x\in\mathbb{C}^\times$ with $x^4\neq 1$, and define 
$\alpha:=\frac{1-x^{2}}{2}$.  Consider the following Yang-Baxter operators:
\[
R :=
\begin{bmatrix}
1 & 0 & 0 & 0\\
0 & 1-x^2 & x^2 & 0\\
0 & 1 & 0 & 0\\
0 & 0 & 0 & -x^2
\end{bmatrix},
\qquad
S=
\begin{bmatrix}
\alpha+x & 0 & 0 & \alpha \\
0 & \alpha & 1-\alpha & 0 \\
0 & 1-\alpha & \alpha & 0 \\
\alpha & 0 & 0 & \alpha-x
\end{bmatrix}.
\]
Let $\rho_n^R,\rho_n^S : B_n \to \mathrm{GL}(2^n)$ be the braid group
representations determined by $R$ and $S$, respectively.  
Then $\rho_n^R$ and $\rho_n^S$ are equivalent for all $n\ge 1$.
{In particular, the eight-vertex constant YBO is $\infty$-equivalent to a type-a charge-conserving solution.} \end{proposition}
\medskip

    Before proceeding to the proof, we mention that the enzuigiri method, applied to $S$ shows that $S$ is not locally CC: the only $M$ such that $M\otimes I+I\otimes M$ that commute with $S$ are of the form $mI$.  Of course this follows from \cite{Hietarinta92} as well.

\begin{proof}
Set
\(
 a_n := x^{\,n-2}.
\)
We define two sequences of matrices $\{D_n\}_{n\ge 0}$ and
$\{P_n\}_{n\ge 0}$ recursively
as follows.
Let
\[
D_0 = (x), \qquad P_0 = (0),
\]
and for $n\ge 1$ define $2\times 2$-block matrices
\begin{equation}   \label{eq:DPrecursion}
D_n :=
\begin{bmatrix}
D_{n-1} &
\displaystyle \frac{x-1}{x+1}\, a_n P_{n-1}\\[6pt]
\displaystyle -\frac{x-1}{x+1}\, P_{n-1} &
a_n D_{n-1}
\end{bmatrix},
\qquad
P_n :=
\begin{bmatrix}
P_{n-1} & a_n D_{n-1}\\
D_{n-1} & -a_n P_{n-1}
\end{bmatrix}.
\end{equation}

We claim that for all $n\ge 1$ $D_n$ and $P_n$ are intertwiners for $\rho_n^R$ and $\rho_n^S$, explicitly:
\[
D_n \; \rho_n^R(\beta) \; = \;  \rho_n^S(\beta) \; D_n,
\qquad
P_n \rho_n^R(\beta) = \rho_n^S(\beta) P_n
\quad
\text{for all }\beta\in B_n.
\]
{We will prove the claim (and hence the Proposition) by induction on $n$, with base $n=2$.}

\medskip

\noindent\emph{Base case.}
For $n=1$, 
the matrices $D_1=\mathrm{diag}(x,1)$ and
$P_1=\begin{bmatrix}0&1\\ x&0\end{bmatrix}$ are invertible, 
{the braid group $B_1$ is trivial, the representations do not use $R,S$,}
and the
intertwining relations are immediate.
For $n=2$, the claim reduces to checking that $D_2$ and $P_2$ satisfy
\[
X R = S X.
\]
Here 
\[
D_{2}=
\begin{bmatrix}
x&0&0&\dfrac{x-1}{x+1}\\[6pt]
0&1&\dfrac{x(x-1)}{x+1}&0\\[6pt]
0&-\dfrac{x-1}{x+1}&x&0\\[6pt]
-\dfrac{x(x-1)}{x+1}&0&0&1
\end{bmatrix},
\qquad
P_{2}=
\begin{bmatrix}
0&1&x&0\\
x&0&0&1\\
x&0&0&-1\\
0&1&-x&0
\end{bmatrix}.
\]

 Now \[
D_{2}R=SD_{2}=
\begin{bmatrix}
x&0&0&\dfrac{x^{2}(1-x)}{x+1}\\[6pt]
0&\dfrac{1-x^{3}}{x+1}&x^{2}&0\\[6pt]
0&x^{2}-x+1&\dfrac{x^{2}(1-x)}{x+1}&0\\[6pt]
\dfrac{x(1-x)}{x+1}&0&0&-x^{2}
\end{bmatrix}, 
\] and \[
P_{2}R=SP_{2}=
\begin{bmatrix}
0&-x^{2}+x+1&x^{2}&0\\
x&0&0&-x^{2}\\
x&0&0&x^{2}\\
0&-x^{2}-x+1&x^{2}&0
\end{bmatrix}.
\]
\medskip

Assume the 
{claim}
holds for $n-1\ge 1$.
{For the inductive step it is enough to show that both $D_n$ and $P_n$ intertwine the images of each of a generating set of braids.}

For $i>1$, the {Artin} braid generator $\sigma_i$ acts as
\[
\rho_n^R(\sigma_i)=I_2\otimes \rho_{n-1}^R(\sigma_{i-1}),
\qquad
\rho_n^S(\sigma_i)=I_2\otimes \rho_{n-1}^S(\sigma_{i-1}).
\]
Since $D_n$ and $P_n$ are defined using a $2\times 2$ block structure
with entries
{that are scalar multiples of}
$D_{n-1}$ and $P_{n-1}$, the induction hypothesis
implies that both $D_n$ and $P_n$ intertwine $\rho_n^R(\sigma_i)$ and
$\rho_n^S(\sigma_i)$ for all $i>1$.  

Explicitly:
Fix $\beta\in B_{n-1}$.  Suppose that $A,B,C,E$ are scalar multiples of $D_{n-1}$ or $P_{n-1}$ so that they intertwine
$\rho_{n-1}^R$ and $\rho_{n-1}^S$ at $\beta$, i.e.
\(
A\,\rho_{n-1}^R(\beta)=\rho_{n-1}^S(\beta)\,A,\)
etc.

Let
\[
M=\begin{bmatrix}A&B\\ C&E\end{bmatrix},
\qquad
R_\beta:=\rho_{n-1}^R(\beta),\quad S_\beta:=\rho_{n-1}^S(\beta).
\]
Now
\[
(I_2\otimes \rho_{n-1}^R(\beta))=
\begin{bmatrix}R_\beta&0\\ 0&R_\beta\end{bmatrix},
\qquad
(I_2\otimes \rho_{n-1}^S(\beta))=
\begin{bmatrix}S_\beta&0\\ 0&S_\beta\end{bmatrix}.
\]
Therefore,
\[
M\,(I_2\otimes \rho_{n-1}^R(\beta))
=
\begin{bmatrix}A&B\\ C&E\end{bmatrix}
\begin{bmatrix}R_\beta&0\\ 0&R_\beta\end{bmatrix}
=
\begin{bmatrix}A R_\beta& B R_\beta\\ C R_\beta& E R_\beta\end{bmatrix},
\]
while
\[
(I_2\otimes \rho_{n-1}^S(\beta))\,M
=
\begin{bmatrix}S_\beta&0\\ 0&S_\beta\end{bmatrix}
\begin{bmatrix}A&B\\ C&E\end{bmatrix}
=
\begin{bmatrix}S_\beta A& S_\beta B\\ S_\beta C& S_\beta E\end{bmatrix}.
\]
Hence 
\[
M\,(I_2\otimes \rho_{n-1}^R(\beta))=(I_2\otimes \rho_{n-1}^S(\beta))\,M
\]
is equivalent, block-by-block, to the four intertwining identities
\[
A R_\beta=S_\beta A,\qquad B R_\beta=S_\beta B,\qquad
C R_\beta=S_\beta C,\qquad E R_\beta=S_\beta E,
\]
which hold by assumption.  In particular, any $2\times 2$ block matrix $M$ whose blocks
each intertwine $\rho_{n-1}^R(\beta)$ and $\rho_{n-1}^S(\beta)$ automatically intertwines
$I_2\otimes \rho_{n-1}^R(\beta)$ and $I_2\otimes \rho_{n-1}^S(\beta)$.

It therefore suffices to verify the intertwining relations for 
{the remaining generator,}
$\sigma_1$.

\medskip

Write
\[
R_{\sigma_1} := \rho_n^R(\sigma_1) = R \otimes I_{2^{n-2}},
\qquad
S_{\sigma_1} := \rho_n^S(\sigma_1) = S \otimes I_{2^{n-2}},
\]

and view these as $4\times 4$ block matrices whose entries are scalar
multiples of $I_{2^{n-2}}$.
Using the recursive definitions, the matrices $D_n$ and $P_n$ may also be
written as $4\times 4$ block matrices whose entries are scalar multiples
of $D_{n-2}$ or $P_{n-2}$.

\medskip 
Fix $n\ge 2$ and set
\[
m:=2^{n-2},\qquad I:=I_m,\qquad D:=D_{n-2},\qquad P:=P_{n-2}.
\]
Write also 
\[
\alpha:=\frac{1-x^2}{2},\qquad 
a_k:=x^{k-2} ,\qquad
h_k:=\frac{x-1}{x+1}\,a_k,\qquad c:=-\frac{x-1}{x+1}.
\]

The braid generator
$\sigma_1$ acts by $R$ (resp.\ $S$) on the first two tensor factors 
{(in the language of tensor space representations, where $\rho_n$ acts on $V^{\otimes n}$ with $V=\C^2$ here)}
and by $I$ on the rest, hence
\[
R_{\sigma_1}:=\rho_n^R(\sigma_1)=R\otimes I
=
\begin{bmatrix}
1\cdot I&0&0&0\\
0&(1-x^2)I&x^2 I&0\\
0&I&0&0\\
0&0&0&(-x^2)I
\end{bmatrix},
\]
and
\[
S_{\sigma_1}:=\rho_n^S(\sigma_1)=S\otimes I
=
\begin{bmatrix}
(\alpha+x)I&0&0&\alpha I\\
0&\alpha I&(1-\alpha)I&0\\
0&(1-\alpha)I&\alpha I&0\\
\alpha I&0&0&(\alpha-x)I
\end{bmatrix}.
\]

By the recursion 
(\ref{eq:DPrecursion}) 
(with $h_n=\frac{x-1}{x+1}a_n$ and $c=-\frac{x-1}{x+1}$),
\[
D_n=\begin{bmatrix} D_{n-1} & h_n P_{n-1}\\ c\,P_{n-1} & a_n D_{n-1}\end{bmatrix},
\qquad
P_n=\begin{bmatrix} P_{n-1} & a_n D_{n-1}\\ D_{n-1} & -a_n P_{n-1}\end{bmatrix}.
\]
Expanding $D_{n-1}$ and $P_{n-1}$ once more in terms of $D=D_{n-2}$ and $P=P_{n-2}$ gives:

\[
D_{n-1}=
\begin{bmatrix}
D & h_{n-1}P\\
cP & a_{n-1}D
\end{bmatrix},
\qquad
P_{n-1}=
\begin{bmatrix}
P & a_{n-1}D\\
D & -a_{n-1}P
\end{bmatrix}.
\]

Therefore $D_n$ is a $4\times 4$ block matrix with $m\times m$ blocks (rows/cols indexed by
$\{00,01,10,11\}$ on the first two tensor factors):
\begin{equation}\label{star_D}
D_n=
\begin{bmatrix}
D & h_{n-1}P & h_n P & h_n a_{n-1}D\\
cP & a_{n-1}D & h_n D & -h_n a_{n-1}P\\
cP & c a_{n-1}D & a_n D & a_n h_{n-1}P\\
cD & -c a_{n-1}P & a_n cP & a_n a_{n-1}D
\end{bmatrix}.
\end{equation}

Similarly
\begin{equation}\label{star_P}
P_n=
\begin{bmatrix}
P & a_{n-1}D & a_n D & a_n h_{n-1}P\\
D & -a_{n-1}P & a_n cP & a_n a_{n-1}D\\
D & h_{n-1}P & -a_n P & -a_n a_{n-1}D\\
cP & a_{n-1}D & -a_n D & a_n a_{n-1}P
\end{bmatrix}.
\end{equation}
Define $\Delta=D_nR_{\sigma_1}-S_{\sigma_1}D_n.$  Then the 16 entries are:

\begin{enumerate}
\item $\displaystyle
\Delta_{11}
=
\bigl(1-(\alpha+x)-\alpha c\bigr)D$
\item $\displaystyle
\Delta_{12}
=
\bigl((1-x^2)h_{n-1}+h_n-(\alpha+x)h_{n-1}+\alpha c a_{n-1}\bigr)P$
\item $\displaystyle
\Delta_{13}
=
\bigl(x^2 h_{n-1}-(\alpha+x)h_n-\alpha c a_n\bigr)P$
\item $\displaystyle
\Delta_{14}
=
a_{n-1}\bigl(-x^2 h_n-(\alpha+x)h_n-\alpha a_n\bigr)D$

\item $\displaystyle
\Delta_{21}
=
\bigl(c-\alpha c-(1-\alpha)c\bigr)P$
\item $\displaystyle
\Delta_{22}
=
\bigl((1-x^2)a_{n-1}+h_n-\alpha a_{n-1}-(1-\alpha)c a_{n-1}\bigr)D$
\item $\displaystyle
\Delta_{23}
=
\bigl(x^2 a_{n-1}-\alpha h_n-(1-\alpha)a_n\bigr)D$
\item $\displaystyle
\Delta_{24}
=
a_{n-1}\bigl(x^2 h_n+\alpha h_n-(1-\alpha)a_n h_{n-1}\bigr)P$

\item $\displaystyle
\Delta_{31}
=
\bigl(c-(1-\alpha)c-\alpha c\bigr)P$
\item $\displaystyle
\Delta_{32}
=
\bigl((1-x^2)c a_{n-1}+a_n-(1-\alpha)a_{n-1}-\alpha c a_{n-1}\bigr)D$
\item $\displaystyle
\Delta_{33}
=
\bigl(x^2 c a_{n-1}-(1-\alpha)h_n-\alpha a_n\bigr)D$
\item $\displaystyle
\Delta_{34}
=
a_{n-1}\bigl(-x^2 a_n h_{n-1}+(1-\alpha)h_n-\alpha h_{n-1}\bigr)P$

\item $\displaystyle
\Delta_{41}
=
\bigl(c-\alpha-(\alpha-x)c\bigr)D$
\item $\displaystyle
\Delta_{42}
=
\bigl(-(1-x^2)c a_{n-1}+a_n c-\alpha h_{n-1}+(\alpha-x)c a_{n-1}\bigr)P$
\item $\displaystyle
\Delta_{43}
=
\bigl(-x^2 c a_{n-1}-\alpha h_n-(\alpha-x)a_n c\bigr)P$
\item $\displaystyle
\Delta_{44}
=
a_{n-1}\bigl(-x^2 a_n-\alpha h_n-(\alpha-x)a_n\bigr)D$
\end{enumerate}

Substituting $\alpha=\tfrac12(1-x^2)$, $a_n=x a_{n-1}$,
$h_n=\tfrac{x-1}{x+1}a_n$, and $c=-\tfrac{x-1}{x+1}$,
each coefficient vanishes identically. Hence $D_nR_{\sigma_1}=S_{\sigma_1}D_n$.  A similar calculation shows that $P_nR_{\sigma_1}=S_{\sigma_1}P_n.$

This completes the induction and proves the claim, observing  that $R$ is precisely a type-a charge conserving solution (in the terminology of \cite{MartinRowell}).
\end{proof}

{Given the far-from-trivial nature of this proof,}
We would feel sheepish indeed if the $\infty$-equivalence could be established by means of a  
(few-move)
 DSL equivalence 
{between the constant Yang--Baxter operators $R$ and $S$ above.}
This is not the case, at least for generic values of $x$.  Indeed the only matrices of the form $T\otimes T$ that commute with $R$ are diagonal, and the same is true of $DS_T(R)$: any $M$ with $M\otimes M$ centralising $DS_T(R)$ is diagonal.  Thus the $DS$-equivalence class of $R$ consists of only of charge conserving YBOs. It is then a straightforward calculation to verify that there is no matrix $A$ such that $A\otimes A$ conjugates $DS_T(R)$ to $S$ with $T$ diagonal.  In particular we see that there are $\infty$-equivalences that are do not come from DSL-equivalences.

\subsection{A Symmetry of Charge-Conserving Solutions: X-equivalence.}
One of the hallmarks of charge-conserving YBOs is that, as a set, they are invariant under conjugation by a diagonal matrix, as noted in \cite{MartinRowell}.  That is, if $R$ is a charge-conserving YBO in rank $N$, and $X$ is an arbitrary $N^2\times N^2$ matrix then $XRX^{-1}$ is also a charge-conserving YBO. In \emph{loc. cit.} this is called \textbf{$X$-equivalence} and is employed to reduce the classification problem. 
Does this mean that $R$ and $XRX^{-1}$ are $\infty$-equivalent?  
In \cite[Conjecture 7.22]{MRT25} it is conjectured that this is so, and a proof for $N=2$ is supplied.  We  prove this in full generality:
\begin{theorem}\label{thm:Ximpliesinfty}  $X$-equivalence implies $\infty$-equivalence (for CC YBOs).
\end{theorem}
\begin{proof}
    Let $R$ be a CC YBO and $X$ an $N^2\times N^2$ diagonal matrix, defined by $X\ket{ab}=x_{a,b}\ket{ab}$. Set $S=XRX^{-1}$. For $a\neq b$, we have $R\ket{ab}=\alpha_{ab}\ket{ab}+\beta_{ab}\ket{ba}$ so \[S\ket{ab}=\alpha_{ab}\ket{ab}+\beta_{ab}\frac{x_{ba}}{x_{ab}}\ket{ba},\] while on $\ket{aa}$ $R$ and $S$ coincide. We must show that $\rho_S$ and $\rho_R$ are equivalent $B_n$ representations for all $n$.  Define a diagonal matrix $X_n$ by \[X_n\ket{i_1\cdots i_n}=\left(\prod_{p<q}x_{i_p,i_q}\right)\ket{i_1\cdots i_n},\] noting that $X_2=X$. We claim that $X_n$ intertwines $\rho_R$ and $\rho_S$.  We must check that \[I^{\otimes k-1}\otimes S\otimes I^{\otimes n-k-1}=X_n(I^{\otimes k-1}\otimes R\otimes I^{\otimes n-k-1})X_n^{-1}.\]  We compute:
    \begin{align*}      
 && X_n(I^{\otimes k-1}\otimes R\otimes I^{\otimes n-k-1})X_n^{-1}\ket{i_1\cdots i_n}=X_n(I^{\otimes k-1}\otimes R\otimes I^{\otimes n-k-1})\left(\prod_{p<q}x_{i_p,i_q}\right)^{-1}\ket{i_1\cdots i_n}\\ 
&& =X_n\left(\prod_{p<q}x_{i_p,i_q}\right)^{-1}\ket{i_1\cdots i_{k-1}}\otimes (\alpha_{i_k,i_{k+1}}\ket{i_ki_{k+1}}+\beta_{i_k,i_{k+1}})\ket{i_{k+1}i_k}\otimes \ket{i_{k+2}\cdots i_n}=\\
 &&\alpha_{i_k,i_{k+1}}\ket{i_1\cdots i_n}+\beta_{i_k,i_{k+1}}\left(\prod_{p<q}x_{i_p,i_q}\right)^{-1}X_n\ket{i_1\cdots i_{k-1}i_{k+1}i_ki_{k+2}\cdots i_n}=\\
&& \alpha_{i_k,i_{k+1}}\ket{i_1\cdots i_n}+\beta_{i_k,i_{k+1}}\frac{x_{i_{k+1},i_k}}{x_{i_k,i_{k+1}}}\ket{i_1\cdots i_{k-1}i_{k+1}i_ki_{k+2}\cdots i_n}=(I^{\otimes k-1}\otimes S\otimes I^{\otimes n-k-1})\ket{i_1\cdots i_n}.
 \end{align*}
 This proves the theorem.
\end{proof}

We conclude this section with an example in which two charge-conserving YBOs are $X$-equivalent, and hence $\infty$-equivalent, without being DSL-equivalent.  This illustrates that $X$-equivalence goes beyond DSL-equivalence.
\begin{example}
Let \[S=
\begin{bmatrix}
1 & 0 & 0 & 0 & 0 & 0 & 0 & 0 & 0 \\
0 & 0 & 0 & 1 & 0 & 0 & 0 & 0 & 0 \\
0 & 0 & 0 & 0 & 0 & 0 & 1 & 0 & 0 \\
0 & 1 & 0 & 0 & 0 & 0 & 0 & 0 & 0 \\
0 & 0 & 0 & 0 & 1 & 0 & 0 & 0 & 0 \\
0 & 0 & 0 & 0 & 0 & 0 & 0 & 1 & 0 \\
0 & 0 & 1 & 0 & 0 & 0 & 0 & 0 & 0 \\
0 & 0 & 0 & 0 & 0 & 1 & 0 & 0 & 0 \\
0 & 0 & 0 & 0 & 0 & 0 & 0 & 0 & 1
\end{bmatrix}
\quad \text{and} \quad
R=
\begin{bmatrix}
1 & 0 & 0 & 0 & 0 & 0 & 0 & 0 & 0 \\
0 & 0 & 0 & -1 & 0 & 0 & 0 & 0 & 0 \\
0 & 0 & 0 & 0 & 0 & 0 & -1 & 0 & 0 \\
0 & -1 & 0 & 0 & 0 & 0 & 0 & 0 & 0 \\
0 & 0 & 0 & 0 & 1 & 0 & 0 & 0 & 0 \\
0 & 0 & 0 & 0 & 0 & 0 & 0 & -1 & 0 \\
0 & 0 & -1 & 0 & 0 & 0 & 0 & 0 & 0 \\
0 & 0 & 0 & 0 & 0 & -1 & 0 & 0 & 0 \\
0 & 0 & 0 & 0 & 0 & 0 & 0 & 0 & 1
\end{bmatrix}.
\]  
 Firstly, $R$ and $S$ are DSL-inequivalent:
{A computer computation reveals that}
the invertible matrices $T$ such that $T\otimes T$ commutes with $R$ are monomial. Moreover, $DS_T(R)$, for any such monomial $T$ has the same property: $[DS_T(R),M\otimes M]=0$ with $M$ invertible implies $M$ is monomial.  Thus,  it suffices to check if $DS_T(R)$ and $S$ are locally equivalent for some choice of $T$. A computer calculation shows they are not.

On the other hand, the matrix $X:=diag(1, -1, -1, 1, 1, -1, 1, 1, 1)$ provides an $X$-equivalence between $R$ and $S$. 
\end{example}

\section{A Whisker Beyond Charge-Conserving}

Charge-conserving braid representations \cite{MartinRowell} are 
{strikingly} 
ubiquitous when viewed through the lens of $\infty$-equivalence.  
In Hietarinta's \cite{Hietarinta92} classification of 
$4\times 4$ YBOs up to local equivalence, scalar multiplication and transpose \cite{Hietarinta92}, there are 11 classes of (varieties of) solutions.  \footnote{Although transpose does not generally correspond to an $\infty$-equivalence, see \cite{Korepinetal}, where they show that, for all but one of Hietarinta's classes, the transposed solutions are locally equivalent to one of the other representatives possibly after parameter changes.}
Several are already charge-conserving
{in Hietarinta's chosen representative}, 
while another handful are just a whisker away from a charge-conserving YBO, 
being of the form 
called \emph{charge-conserving-with-glue} (\CCWG) in \cite{AlmateariMartinRowell}.  {Conceptually, one chooses an ordering on the charges spaces $V^{\mset{i_1,\ldots,i_k}}$, and then a \CCWG\ matrix may have couplings (``glue") between two  charge spaces if the second is lower in the order than the first.  So while a CC matrix preserves a grading by charge spaces, a \CCwg\ matrix preserves a filtration.}  {Whenever $R+M$ is a \CCWG\ solution, $R$ itself is a CC solution, and $R+M$ has the same characteristic polynomial as $R$.}  Of the 11, only 3 
{of Hietarinta's chosen representatives}
are not of this type.  
One of these has been shown to be an ill-disguised charge-conserving solution in \cite[Paragraph (3.1)]{AlmateariMartinRowell}, i.e., it is $\infty$-equivalent, via a $DS$-move, to a charge-conserving (monomial) solution. 
{The remaining two are 8-vertex (shown 
$\infty$-equivalent to 
CC in our main result, \ref{prop:8va} above) 
and so finally, Ising (given by $E$ in \ref{pr:Ising} below). } 

\begin{proposition}  \label{pr:Ising}
(I)    {The} solution 
\[
E:=\frac{1}{\sqrt{2}}\left[ \begin {array}{cccc} 1&0&0&1\\ \noalign{\medskip}0&1&1&0
\\ \noalign{\medskip}0&-1&1&0\\ \noalign{\medskip}-1&0&0&1\end {array}\right]
\] 
is not $\infty$-equivalent to any scalar multiple or transpose of 
any solution from 
Hietarinta's other 10 classes
of varieties of solutions. 

(II) The solution $E$ 
is not $\infty$-equivalent to a \CCWG\ solution. 
\end{proposition}

    Combining with Proposition \ref{prop:8va}, we have the following  as a corollary:
    
    \begin{theorem}\label{thm:uniqueinftyclassinrank2}
        Up to transpose symmetry, the solution $E$ above represents the only (projective) $\infty$-class of rank $2$ solutions that is not  $\infty$-equivalent to a \CCWG\  solution.    
    \end{theorem}

\begin{proof}(of Proposition \ref{pr:Ising})
  We note that $E$ (and its transpose) has eigenvalues $\zeta_8,\overline{\zeta_8}$ each of multiplicity $2$, where $\zeta_8=e^{2\pi i/8}$.  
  Moreover, $E$ is diagonalisable, indeed unitary so \CCWG\ is reduced to CC.  
  This is enough to eliminate all but one form: \[\left[ \begin {array}{cccc} a&0&0&0\\ \noalign{\medskip}0&a+b&-b&0
\\ \noalign{\medskip}0&a&0&0\\ \noalign{\medskip}0&0&0&b\end {array}
 \right]
\] 
(in either classification \cite{Hietarinta92,MartinRowell})
which has eigenvalues $a,b$ each with multiplicity $2$.  Thus one must have $a=1/b=\zeta_8$, by symmetry.  But this has been shown to be not $\infty$-equivalent
to $E$ in \cite[Paragraph (7.2)]{MRT25}, 
by comparing $Trace(\rho_3(\sigma_1 \sigma_2^{-1}))$ for the two representations.  
Since multiplying by a scalar $c$ does not change this trace, we are done.
\end{proof}

Lacking a rank $3$ classification, we cannot 
yet 
characterize the $\infty$-equivalences classes that do not contain a \CCWG\ representative in that case. However, we can produce examples of rank $3$ solutions not in the $\infty$-class of \CCWG\ solutions.  The method of proof is to use the spectra of \CCWG\ solutions: any \CCWG\ solution has the same spectrum as the underlying CC shadow \cite{AlmateariMartinRowell}.  On the other hand, the spectra of all rank $3$ CC solutions are calculated in \cite[Section A5]{MartinRowell}.  

    A first example is from \cite[Section 4.1]{HMR}, which has two eigenvalues, one of multiplicity 8 and one of multiplicity 1.  We see that this is not possible for any CCwg solution, since it cannot happen for CC solutions.

    \begin{example}\label{ex:ACCexample} In \cite{HMR} the following (\emph{additive charge conserving}) solution is found:
       \[ R=\left[
\begin{array}{ccccccccc}
1 & \cdot & \cdot & \cdot & \cdot & \cdot & \cdot & \cdot & \cdot 
\\
 \cdot & 1 & \cdot & \cdot & \cdot & \cdot & \cdot & \cdot & \cdot 
\\
 \cdot & \cdot & a  & \cdot & {x_1}  & \cdot & b  & \cdot & \cdot 
\\
 \cdot & \cdot & \cdot & 1 & \cdot & \cdot & \cdot & \cdot & \cdot 
\\
 \cdot & \cdot & \frac{{x_3} \left(a -1\right)}{b} & \cdot & \frac{{x_1} {x_3} +b}{b} & \cdot & {x_3}  & \cdot & \cdot 
\\
 \cdot & \cdot & \cdot & \cdot & \cdot & 1 & \cdot & \cdot & \cdot 
\\
 \cdot & \cdot & \frac{{x_3}^{2} {x_1}^{2}}{b^{3}} & \cdot & -\frac{{x_1} \left(a b +{x_1} {x_3} \right)}{a \,b^{2}} & \cdot & -\frac{{x_3} {x_1}}{a b} & \cdot & \cdot 
\\
 \cdot & \cdot & \cdot & \cdot & \cdot & \cdot & \cdot & 1 & \cdot 
\\
 \cdot & \cdot & \cdot & \cdot & \cdot & \cdot & \cdot & \cdot & 1 
\end{array}\right]\] where $\cdot$ represents $0$ and the parameters satisfy \[
x_3^{2} x_1^{2} a 
+ a^{2} b^{2}
+ x_1  b  x_3 a 
- a \,b^{2}
- x_1 b x_3
\;\; = \;\;
b^2 a (a-1) + b x_3 x_1 (a-1) + a x_3^2 x_1^2.
\]
The eigenvalues are $1$ (multiplicity 8) and $- \left(\frac{x_1 x_3 }{b}\right)^2$ (once).  From \cite[Paragraph (A.5)]{MartinRowell} we deduce that no (specialisation of generic) rank $3$ charge conserving solution may have these eigenvalue multiplicities.  In detail, those solutions corresponding to single 
bi-coloured Young diagram generically have 2 eigenvalues but never with multiplicities 8 and 1, while the remaining solutions all have eigenvalues of the form $\pm \lambda$, which is generically impossible for the $R$ above.
    \end{example}

\section{When $\infty$-Equivalence is Just Too Much}

{In this section we recall several notions of equivalence weaker than \infeq, and compare them. We zero in on one that has a categorical manifestation and illustrate its nuances with an example. }

\medskip

An ordinary  (not necessarily monoidal) functor  $F: \Braid \rightarrow \Mat$ is essentially the same thing as 
an arbitrary sequence  $\rho = \{ \rho_n \}_n$  of representations of braid groups,  
i.e. more explicitly, a pair of sequences 
$( \{N_n \in \N\}_n ,\,  \{  \rho_n : B_n \rightarrow \Mat(N_n,N_n) \}_{n} )$. 
%
As noted for example in \cite{MRT25}, 
this 
kind of functor 
is too loose a notion to require an interesting solution. 

{In this section we consider sequences more strongly constrained than an arbitrary sequence, but  potentially  less strongly constrained than a full-on braid representation.}

\medskip 

{Firstly,  although arbitrary functors/sequence are too 
loosely constrained
to be interesting, 
various easy (to define) notions of equivalence can be given already at this level 
including,  note,  $\infty$-equivalence. 
Some further useful equivalences now follow. 
\\
}{For a sequence $ \rho = \{ \rho_n\}_n$ we define $A_n$ as the image of $\C B_n$ under $\rho_n$, that is, $A_n = \C B_n / \ker \rho_n$. 
}

{Two sequences $\rho, \rho'$ are: 
\\
{\bf{kernel equivalent}} if $\ker \rho_n = \ker \rho_n'$ for all $n$. 
\\
{\bf{spectrum equivalent}} if for every element $X \in \C B_n$ the spectrum of $\rho_n(X)$ 
equals that of $\rho'_n(X) $, including multiplicities, for all $n$.  This is the same as asking for the characters of $\rho_n$ and $\rho'_n$ to agree for all $n$ (use Newton identities to recover characteristic polynomials from the characters, the other direction is clear).
\\
{\bf{Morita-Schur-Weyl equivalent}} if 
the indecomposable content of 
$\rho_n$ and $\rho'_n $ differ only up to non-zero multiplicities.
\\
{\bf{mod-radical equivalent}} if the simple composition factors 
of $\rho_n$ and $\rho'_n$ 
(and their multiplicities) 
are the same for each $n$. 
\\
{\bf{weakly mod-radical equivalent}} if $A_n/J(A_n) \cong A'_n/J(A'_n)$ for each $n$, where $J(\mathscr{A})$ is the Jacobson radical of an algebra $\mathscr{A}$.
\\
}

{Note that  
$\infty$-equivalence implies spectrum equivalence, which is the same thing as mod-rad equivalence.  Indeed, we can define mod-rad $m$-equivalence to be the requirement that the simple composition factors of $\rho_n$ and $\rho'_n$ and their multiplicities are the same for all $n\leq m$.  Then the proof of Theorem \ref{thm:m-implies-infty-for-ss-rank-N} yields: \emph{for each $N$, there is an $m(N)$ such that mod-rad $m(N)$-equivalence implies mod-rad equivalence.}  This is the appropriate non-semisimple generalization of Theorem \ref{thm:m-implies-infty-for-ss-rank-N}--we do not currently know if the semisimplicity assumption can be dropped.

Also $\infty$-equivalence implies  
MSW equivalence implies kernel equivalence. 
But kernel equivalence coincides with MSW equivalence 
if representations are semisimple (as for example if they are unitary). }

\begin{figure}[htbp]
\centering
\resizebox{\linewidth}{!}{%
\begin{tikzpicture}[
  x=1cm,y=1cm,
  >={Stealth[length=2.2mm,width=1.5mm]},
  implication/.style={-{Stealth},double,very thick,draw=black!78},
  relation/.style={
    draw=black!68,
    rounded corners=2pt,
    very thick,
    fill=blue!4,
    align=center,
    text width=2.85cm,
    minimum height=1.18cm,
    inner sep=5pt,
    font=\small
  },
  easy/.style={relation,fill=green!7},
  natural/.style={relation,fill=blue!9,text width=3.2cm},
  finite/.style={relation,fill=orange!8,text width=2.7cm},
  morita/.style={relation,fill=violet!8,text width=3.25cm},
  radical/.style={relation,fill=cyan!7,text width=3.7cm},
  note/.style={
    draw=black!40,
    rounded corners=2pt,
    fill=black!2,
    align=left,
    text width=5.0cm,
    inner sep=5pt,
    font=\scriptsize
  },
  lab/.style={font=\scriptsize,fill=white,inner sep=1.3pt,text=black!75}
]

\node[easy] (local) at (-3.0,5.8)
  {\textbf{Local equivalence}};

\node[easy] (ds) at (3.0,5.8)
  {\textbf{DS-equivalence}};

\node[easy,text width=3.25cm] (dsl) at (0,3.75)
  {\textbf{DSL-equivalence}};

\node[natural] (infty) at (0,1.65)
  {\textbf{$\boldsymbol{\infty}$-equivalence}};

\node[finite] (three) at (-4.4,-0.65)
  {\textbf{$3$-equivalence}};

\node[finite] (two) at (-4.4,-2.55)
  {\textbf{$2$-equivalence}};

\node[finite] (one) at (-4.4,-4.45)
  {\textbf{$1$-equivalence}\\[-1pt]
   same rank};

\node[radical] (radical) at (0,-0.65)
  {\textbf{Mod-rad equivalence} \\[-1pt]
  spectrum equivalence};

\node[morita] (msw) at (4.85,-0.65)
  {\textbf{MSW-equivalence}};

\node[morita] (brsys) at (4.85,-2.55)
  {\textbf{BrSys-equivalence}\\[-1pt]
   kernel equivalence
  };

\node[radical] (weakradical) at (0,-2.55)
  {\textbf{Weak mod-rad equivalence}};

\draw[implication] (local) -- (dsl);
\draw[implication] (ds) -- (dsl);
\draw[implication] (dsl) -- (infty);
\draw[implication] (infty) -- (three);
\draw[implication] (infty) -- (radical);
\draw[implication] (infty) -- (msw);
\draw[implication] (msw) -- (brsys);
\draw[implication] (radical) -- (weakradical);
\draw[implication] (brsys) -- (weakradical);
\draw[implication] (three) -- (two);
\draw[implication] (two) -- (one);

\end{tikzpicture}%
}
\caption{Implication poset for equivalences of constant Yang--Baxter
operators. Most, if not all, of these arrows are strictly unidirectional. On the other hand, assuming semisimplicity collapses the three equivalences: $\infty$, mod-rad and weak mod-rad, and also implies that MSW and $\BrSys$ equivalences are indistinguishable.}
\label{fig:yb-equivalence-poset}
\end{figure}

\medskip

In \cite{RWlocalization} sequences of unitary $B_n$ representations compatible with inclusions $B_n\rightarrow B_{n+1}$ are studied with the goal of determining when the corresponding towers of $B_n$ representations can be uniformly \emph{localised}--roughly meaning they can be encoded in a unitary solution to the YBE.  This is inspired by the problem in quantum information: \emph{when can a topological model \cite{FLKW} be simulated directly by means of a quantum circuit model?}
A categorical generalisation in terms of \emph{Braided Systems} is found in \cite{RowellBrSys}, as follows:

A braided system is a triple 
\[
\mathbf V=
\bigl(\{V_n\}_{n\geq2},
      \{\rho_n\}_{n\geq2},
      \{\phi_n\}_{n\geq2}\bigr),
\]
where, for every $n\geq2$, 
\begin{enumerate}
\item $V_n$ is a finite-dimensional complex vector space;
\item
\(
\rho_n:\mathbb{C}B_n\longrightarrow\End(V_n)
\)
is a unital algebra homomorphism, and
\(
A_n:=\rho_n(\mathbb{C}B_n)\subseteq\End(V_n)
\)
is its image algebra;
\item
\(
\phi_n:A_n\hookrightarrow A_{n+1}
\)
is an injective unital algebra homomorphism satisfying
\[
\phi_n\circ\rho_n=\rho_{n+1}\circ\iota_n.
\]
\end{enumerate} 
where \(
\iota_n:\mathbb{C}B_n\longrightarrow\mathbb{C}B_{n+1}\)
is the  inclusion given by \(
\iota_n(\sigma_i)=\sigma_{i},
\).
\medskip \\
It will be clear that this is less strongly constrained that a braid representation, but more strongly than an arbitrary sequence.

    \begin{definition}
    

{The category $\BrSys$ has braided systems as objects; and morphisms as follows.}  
%
Let
\(
\mathbf V=(V_n,\rho_n,\phi_n)_{n\geq2},
\; \text{and} \;
\mathbf V'=(V_n',\rho_n',\phi_n')_{n\geq2}
\)
be objects of $\BrSys$, with image algebras
\[
A_n=\rho_n(\mathbb{C}B_n),
\qquad
A_n'=\rho_n'(\mathbb{C}B_n).
\]

A \emph{morphism}
\(
\Psi:\mathbf V\longrightarrow\mathbf V'
\)
is a family of unital algebra homomorphisms
\[
\psi_n:A_n\longrightarrow A_n'
\qquad n\geq2,
\]
such that
\(
\psi_n\circ\rho_n=\rho_n'
\) and 
\(\phi_n'\circ\psi_n=\psi_{n+1}\circ\phi_n.
\)
Composition and identities are defined level-wise.
\end{definition}

{Note that this construction indeed defines a category.}

\medskip 

 The second condition above for morphisms is redundant, but we retain it for our comfort (and to not stray too far from \cite{RWlocalization}).  Indeed, for
$x\in\mathbb{C}B_n$,
\[
\begin{aligned}
\phi_n'\psi_n\rho_n(x)
 &=\phi_n'\rho_n'(x)
 =\rho_{n+1}'(\iota_n(x))\\
 &=\psi_{n+1}\rho_{n+1}(\iota_n(x))
 =\psi_{n+1}\phi_n\rho_n(x).
\end{aligned}
\]
Since $\rho_n$ is surjective onto $A_n$, this proves 
the second condition:
\[
\phi_n'\psi_n=\psi_{n+1}\phi_n.
\]
Thus a morphism may equivalently be defined using only
$\psi_n\rho_n=\rho_n'$.

Note that this implies that $\BrSys$ is \emph{thin}, in the sense that the set of morphisms between $\mathbf V$ and $\mathbf V'$ is either empty or a singleton, since $\psi_n\circ \rho_n=\rho_n'$ determines each $\psi_n$.

\begin{proposition}
{Let $\mathbf V, \mathbf V'$ be as above.}
Set
\(
K_n:=\ker\rho_n\) and 
\(K_n':=\ker\rho_n'.
\)
Then:
\begin{enumerate}
\item a morphism $\mathbf V\to\mathbf V'$ exists if and only if
\[
K_n\subseteq K_n'
\qquad(n\geq2);
\]

\item when it exists, the morphism is unique and every $\psi_n$ is
surjective;

\item the objects $\mathbf V$ and $\mathbf V'$ are isomorphic in
$\BrSys$ if and only if
\[
K_n=K_n'
\qquad(n\geq2).
\]
\end{enumerate}
\end{proposition}

\begin{proof}
Any morphism $\Psi=\{\psi_n\}_{n\geq 2}$ must satisfy
\[
\psi_n\rho_n(x)=\rho_n'(x) \; \text{for all}\; x\in\C B_n \; \text{and} \; n\geq 2
\]
and is therefore uniquely determined.  This formula defines a
well-defined map precisely when
\[
\rho_n(x)=0\quad\Rightarrow\quad\rho_n'(x)=0,
\]
which is equivalent to $K_n\subseteq K_n'$.  The map is surjective
because $\rho_n'$ is surjective onto $A_n'$.

If $\mathbf V\cong\mathbf V'$, applying the preceding criterion to an
isomorphism and its inverse gives
\[
K_n\subseteq K_n'
\quad\text{and}\quad
K_n'\subseteq K_n.
\]
Hence $K_n=K_n'$.  Conversely, equality of the kernels gives mutually
inverse canonical maps
\[
\mathbb{C}B_n/K_n\longleftrightarrow\mathbb{C}B_n/K_n',
\]
and therefore an isomorphism in $\BrSys$.
\end{proof}

This motivates the following more concise:
\begin{definition}
Objects in $\BrSys$ are \textbf{equivalent} if 
\(
\ker\rho_n=\ker\rho_n'
\quad\text{for every }n\geq2
\),
{i.e. if their underlying sequences $\rho$ are kernel equivalent.}
\end{definition}
Observe that any YBO $R$ yields an object in $\BrSys$ taking $V_n=V^{\otimes n}$, 
$\rho_n=\rho_n^R$ 
and $\phi_n(f)=f\otimes I$.
This leads to a looser notion of equivalence than $\infty$-equivalence: we say $R$ and $S$ are \textbf{$\BrSys$-equivalent} if the corresponding objects in $\BrSys$ are equivalent.  

{The question raised in \cite{RWlocalization} is: when is an object in $\BrSys$ associated with an object $X$ in a (unitary) braided fusion category equivalent to an object coming from a (unitary) YBO $R$?  The conjecture is that this is rare, occurring precisely when the quantum dimension of $X$ is the square root of an integer.  Combining with the \emph{property F} conjecture \cite{NaiduRowell} this suggest a no-go theorem: localisability precludes braiding-universality, thus making direct simulation of topological quantum computers and braiding universality incompatible.}

Recall that $\infty$-equivalence {of YBOs $R$ and $S$, say,} gives a family $\{T_n\}_{n\geq 2}$ 
{of isomorphisms}
which are perfectly suitable choices for $\psi_n$.  On the other hand, 
we may have $\BrSys$-equivalent $R$ and $S$ solutions \emph{of different ranks}:

\begin{example}\label{ex:BrSysequiv}
    Consider the solution $R$ of Example \ref{ex:ACCexample} and \[R'=\left[ \begin {array}{cccc} 1&0&0&0\\ \noalign{\medskip}0&1+x&-x&0
\\ \noalign{\medskip}0&1&0&0\\ \noalign{\medskip}0&0&0&1\end {array}
 \right] \] with $x=- \left(\frac{x_1 x_3 }{b}\right)^2.$  After rescaling by $-1$, both $R$ and $R'$ yield faithful representations of the tower of algebras $TL_n(-x)$ see \cite{HMR} and \cite{Martin92} for proofs of these facts.  Therefore $R$ and $R'$ are $\BrSys$-equivalent.
\end{example}

\bibliographystyle{abbrv}
\bibliography{more.bib}  

\begin{thebibliography}{10}

\bibitem{adamek2004abstract}
J.~Ad{\'a}mek, H.~Herrlich, and G.~E. Strecker.
\newblock {\em Abstract and concrete categories. The joy of cats}.
\newblock Citeseer, 2004.

\bibitem{AkgunMerebVendramin2022}
{\"O}.~Akg{\"u}n, M.~Mereb, and L.~Vendramin.
\newblock Enumeration of set-theoretic solutions to the {Yang--Baxter}
  equation.
\newblock {\em Mathematics of Computation}, 91(335):1469--1481, 2022.

\bibitem{AlmateariMartinRowell}
S.~Almateari, P.~Martin, and E.~C. Rowell.
\newblock On higher representation theory via categories of type
  charge-conserving–with–glue.
\newblock {\em arXiv 2601.18452}, 2026.

\bibitem{DS}
A.~Doikou and A.~Smoktunowicz.
\newblock Set-theoretic {Y}ang--{B}axter \& reflection equations and quantum
  group symmetries.
\newblock {\em Letters in Mathematical Physics}, 111:1--40, 2021.

\bibitem{ESS}
P.~Etingof, T.~Schedler, and A.~Soloviev.
\newblock Set-theoretical solutions to the quantum {Y}ang-{B}axter equation.
\newblock {\em Duke Math. J.}, 100(2):169--209, 1999.

\bibitem{FLKW}
M.~H. Freedman, A.~Kitaev, M.~J. Larsen, and Z.~Wang.
\newblock Topological quantum computation.
\newblock {\em Bull. Amer. Math. Soc. (N.S.)}, 40(1):31--38, 2003.

\bibitem{GalindoRowell}
C.~Galindo and E.~C. Rowell.
\newblock Braid representations from unitary braided vector spaces.
\newblock {\em J. Math. Phys.}, 55(6):061702, 13, 2014.

\bibitem{GJ}
D.~M. Goldschmidt and V.~F.~R. Jones.
\newblock Metaplectic link invariants.
\newblock {\em Geom. Dedicata}, 31(2):165--191, 1989.

\bibitem{Gur}
D.~I. Gurevich.
\newblock The {Yang--Baxter} equation and a generalization of formal {Lie}
  theory.
\newblock {\em Soviet Math. Dokl.}, 33(3):758--762, 1986.
\newblock English translation of Dokl. Akad. Nauk SSSR 288 (1986), no. 4,
  797--801.

\bibitem{Hietarinta92}
J.~Hietarinta.
\newblock All solutions to the constant quantum {Yang}–{Baxter} equation in
  two dimensions.
\newblock {\em Physics Letters A}, 165(3):245--251, 1992.

\bibitem{HMR}
J.~Hietarinta, P.~Martin, and E.~C. Rowell.
\newblock Solutions to the constant {Y}ang--{B}axter equation: Additive charge
  conservation in three dimensions.
\newblock {\em Proceedings of the Royal Society A: Mathematical, Physical and
  Engineering Sciences}, 480(2294):20230810, 2024.

\bibitem{LPW}
G.~Lechner, U.~Pennig, and S.~Wood.
\newblock {Yang--Baxter} representations of the infinite symmetric group.
\newblock {\em Advances in Mathematics}, 355:106769, 2019.

\bibitem{maclane2013categories}
S.~{Mac~Lane}.
\newblock {\em Categories for the working mathematician}.
\newblock Springer Science \& Business Media, 2013.

\bibitem{Korepinetal}
S.~Maity, V.~K. Singh, P.~Padmanabhan, and V.~Korepin.
\newblock Algebraic classification of {H}ietarinta's solutions of {Yang-Baxter}
  equations: Invertible {$4\times 4$} operators.
\newblock {\em Journal of High Energy Physics}, 2024(12):67, 2024.

\bibitem{MartinRowell}
P.~Martin and E.~C. Rowell.
\newblock Classification of spin-chain braid representations.
\newblock {\em Communications in Mathematical Physics}, 407:201, 2026.

\bibitem{Martin92}
P.~P. Martin.
\newblock On {S}chur-{W}eyl duality, {$A_n$} {H}ecke algebras and quantum
  {${\rm sl}(N)$} on {$\bigotimes^{n+1}{\bf C}^N$}.
\newblock In {\em Infinite analysis, {P}art {A}, {B} ({K}yoto, 1991)},
  volume~16 of {\em Adv. Ser. Math. Phys.}, pages 645--673. World Sci. Publ.,
  River Edge, NJ, 1992.

\bibitem{MRT25}
P.~P. Martin, E.~C. Rowell, and F.~Torzewska.
\newblock A categorical perspective on braid representations.
\newblock {\em arXiv 2506.07950}, 2025.

\bibitem{NaiduRowell}
D.~Naidu and E.~C. Rowell.
\newblock A finiteness property for braided fusion categories.
\newblock {\em Algebras and Representation Theory}, 14(5):837--855, 2011.

\bibitem{RowellBrSys}
E.~C. Rowell.
\newblock Braiding circuits, localisation, and property {F}.
\newblock In {\em Mini-Workshop: The Yang--Baxter Equation and Representations
  of Braid Groups}, volume~22 of {\em Oberwolfach Reports}, pages 2662--2665.
  EMS Press, 2025.

\bibitem{RWlocalization}
E.~C. Rowell and Z.~Wang.
\newblock Localization of unitary braid group representations.
\newblock {\em Communications in Mathematical Physics}, 311(3):595--615, 2012.

\end{thebibliography}
\end{document}